\documentclass[11pt,a4paper]{amsart}
\usepackage[french, ngerman,italian, english]{babel}
\usepackage{amsmath,amsthm, amssymb, amsfonts}
\usepackage{setspace}
\usepackage{anysize}
\usepackage{url}
\usepackage{bm}
\usepackage{graphicx}
\usepackage[usenames,dvipsnames]{color}
\usepackage[colorlinks=true, linkcolor=Red, citecolor=Green]{hyperref}
\usepackage{paralist}
\usepackage{verbatim}
\usepackage[utf8]{inputenc}
\usepackage[all]{xy}

\theoremstyle{plain}
\newtheorem{theorem}{\bf Theorem}[section]

\newtheorem*{theorem*}{Theorem}
\newtheorem{proposition}[theorem]{\bf Proposition}
\newtheorem{lemma}[theorem]{\bf Lemma}
\newtheorem{corollary}[theorem]{\bf Corollary}
\newtheorem{conjecture}[theorem]{\bf Conjecture}
\newtheorem*{conjecture*}{\bf Conjecture}
\theoremstyle{definition}

\newtheorem{remark}[theorem]{\bf Remark}

\theoremstyle{remark}
\newtheorem{example}[theorem]{\bf Example}
\theoremstyle{example}

\def \aa{{\alpha}}

\def \mfa{\mathfrak{a}}
\def \mfb{\mathfrak{b}}

\def \opc{{\operatorname{c}}}

\def\D{\Delta}

\def \G{{\Gamma}}

\def \HS{{\operatorname{HS}}}

\def \HP{{\operatorname{hp}}}
\def \hci{{\operatorname{hci}}}

\def \link{{\operatorname{link}}}

\def \Mon{{\operatorname{Mon}}}

\def \tr{{\operatorname{tr}}}
\def \ZZ{\mathbb Z}
\def \max{{\operatorname{max}}}

\begin{document}
\title{Determinantal  Schemes and Pure O-sequences}
\author[A. Constantinescu]{Alexandru Constantinescu}
\thanks{Both authors were supported by the Swiss National Science Foundation (Grant number 123393)}
\address{Institut f\"ur Mathematik und Informatik, Freie Universi\"at Berlin, Arnimallee 3, 14195  Berlin, Germany}
\email{aconstant@zedat.fu-berlin.de}
\urladdr{\href{http://userpage.fu-berlin.de/aconstant}{http://userpage.fu-berlin.de/aconstant}} 

\author[M. Mateev]{Matey Mateev}

\address{Institut de Math\'ematiques, Universit\`e de Neuch\^atel B\^atiment UniMail, Rue Emile-Argand 11
2000 Neuch\^atel, Switzerland}
\email{matey.mateev@unine.ch}
\urladdr{\href{http://math.unibas.ch/institut/personen/profil/profil/person/mateev/}{http://math.unibas.ch/institut/personen/profil/profil/person/mateev/}} 
\date{{\small \today}}

\maketitle
\begin{abstract}
 We prove that if a standard determinantal scheme is level, then its $h$-vector is a log-concave  pure O-sequence, and conjecture that the converse also holds.
% 
% 
% We prove that the $h$-vector of a level standard determinantal scheme is a log-concave  pure O-sequence, and conjecture that the converse is also true.
% 
  Among other cases, we prove the conjecture in codimension two, or when the entries of the corresponding degree matrix are positive.  We also  find formulae for the  $h$-vector in terms of the degree matrix.
\end{abstract}

\section*{Introduction}
Classical determinantal rings have made their way from algebraic geometry to commutative algebra more than fifty years ago and have been an active research topic ever since. Over the years, the study has been extended to pfaffian ideals of generic skew-symmetric matrices and to  determinantal ideals of ladders, of symmetric matrices and of homogeneous polynomial matrices. Defining ideals of Segre varieties,  Veronese varieties,  rational normal scrolls and rational normal curves are all examples of such objects.  
 We refer  to the  books of W. Bruns and U. Vetter \cite{BV88},  of R.M. Mir\'o-Roig \cite{MiR08}, and of C. Baetica \cite{Bae06} for overviews of this  vast subject. 

We study the Hilbert functions of standard determinantal rings. %, that is quotients of the polynomial ring by a standard determinantal ideal. 
Ideals  defined by the maximal minors of a homogeneous, polynomial,  $t\times(t+c-1)$ matrix  $M$ are called standard determinantal if they define a scheme of the "expected codimension", i.e. if their height is $c$. These ideals are   Cohen-Macaulay, and a graded minimal free resolution for them is given by the Eagon-Northcott complex \cite{EN62}. Their Hilbert function and their graded Betti numbers  are determined by the degrees of the polynomials in $M$.  Hilbert functions of  determinantal ideals have been studied, among many others, by 
S. Abhyankar \cite{Abh88},
W. Bruns, A. Conca and J. Herzog \cite{CH94, BC03},
S. Ghorpade  \cite{Gho96, Gho02},
N. Budur, M. Casanellas, and E. Gorla \cite{BCG04}.

Our main result (Theorem \ref{thm:main}) states that if in each column of $M$ all polynomials have the same degree, then the $h$-vector of the corresponding standard determinantal ring is a log-concave pure O-sequence.  The idea of the proof is to obtain the  $h$-vectors of such matrices  as  $h$-vectors of some representable matroids, and then use the results of J. Huh \cite{Hu12} for log-concavity, and those of the first author with M. Varbaro \cite{CV3} to prove that they are  pure O-sequences. 
We conjecture that the converse of this theorem also holds, namely if the $h$-vector of a standard determinantal ideal is a pure O-sequence, then all the degrees in each column of its defining matrix must be equal (Conjecture \ref{conj:main}). 

%A sequence of integers $(h_0,\dots,h_s)$ is called log-concave if $h_{i-1}h_{i+1}\le h_i^2$ for all $i$.

 A pure O-sequence is the Hilbert function of some monomial, artinian, level algebra. 
 Equivalently, a pure O-sequence  can be described as the $f$-vector of a pure multicomplex, or of a pure  order ideal. 
In \cite{Hi89}, T. Hibi proved  that if $(h_0,\dots,h_s)$ is a pure O-sequence, then $h_i \le h_{s-i}$ for all $i=0,\dots,\lfloor s/2\rfloor$.
Other than the Hibi inequalities  and some \emph{ad hoc} methods, we are not aware of any criteria which imply non-purity for an O-sequence. In most specific  examples, an exhaustive computer listing of all pure O-sequences with some fixed parameters is needed to check non-purity. Moreover, while a complete characterization of pure O-sequences is considered to "solve all basic problems of design theory" (G. Ziegler \cite[Exercise 8.16]{Zi95}), such a goal is expected to be  "nearly impossible" by several experts (see M. Boij, J. Migliore, R.M. Mir\'o-Roig, U.Nagel, F. Zanello  \cite{BMMNZ}).
The validity of Conjecture \ref{conj:main}, together with the computational formulae we find, would provide a fast  way to construct (for fixed codimension, socle degree and type) large families of  $O$-sequences which are not pure. 

The key to most of our proofs is provided by Lemma \ref{key}. Using a basic double link from Gorenstein liaison theory,
  we describe a recursive formula for the $h$-vector of the standard determinantal ring corresponding to $M$, in terms of $h$-vectors corresponding to submatrices of $M$. Using this lemma we find simple formulae for the length and the last entries of the $h$-vectors (Lemma \ref{lem:length} and Proposition \ref{prop:lastentries}), as well as an explicit formula for the $h$-polynomial for every standard determinantal ring (Proposition \ref{prop:h-poly}). 
  
Using the Eagon-Northcott resolution, we show that a standard determinantal ideal is level (i.e. its socle is concentrated in one degree) if and only if  in each column of $M$ all polynomials have the same degree. 
In the last part of the paper we prove several cases of Conjecture \ref{conj:main}. In particular,  we prove that the statement is true for matrices with all entries of positive degree and for matrices in which the degrees in the second row are strictly smaller than the degrees in the first row. 
Many of the results in this paper have been suggested and double-checked using intensive computer experiments done with CoCoA. \\

The authors thank Elisa Gorla for many helpful discussions and suggestions.

\section{Preliminaries}
We first recall most of the algebraic and geometric notions that we  use, and  then  prove the key lemma of this paper.\\

Let $\Bbbk$ be an infinite field, and $S=\Bbbk[x_{0},\ldots,x_{n}]$ be the polynomial ring over $\Bbbk$.  
For any two integers $t,c \ge 1$, a matrix $M$ of size $t\times (t+c-1)$, with polynomial entries, is called \emph{homogeneous} if it represents a homogeneous map of degree zero between   graded free $S$-modules 
$$\bigoplus_{i=1}^t S(b_i) \xrightarrow{~~M~~} \bigoplus_{j=1}^{t+c-1} S(a_j).$$
Let  $f_{i,j}\in S$ be the entries of $M$. The homogeneity condition implies that $\deg f_{i,j} = a_{j}-b_{i}$ for all $i,j$. Whenever $b_i > a_j$ we have $f_{i,j}=0$, and without loss  of generality we may assume that $M$ does not contain invertible elements (i.e. $f_{i,j}=0$ when $a_j=b_i$). Alternatively, a matrix with polynomial entries is homogeneous if and only if all its minors are homogenous polynomials (if and only if all its 2$\times$2 minors are homogeneous). 
We will denote by $I_{\max}(M)$ the ideal generated by the maximal minors of the matrix $M$. 

An ideal $I\subseteq S$ of height $c$ is a \emph{standard determinantal ideal} if it is generated by the maximal minors of a $t\times (t+c-1)$  homogeneous matrix. As these ideals are saturated, they define a projective  scheme $X\subset\mathbb{P}^n$. We call all such schemes \emph{standard determinantal schemes}.
The matrix $A=\left(a_{i,j}\right)\in\mathbb{Z}^{t\times (t+c-1)}$,
with $a_{i,j}=b_{j}-a_{i}$, is called the \emph{degree matrix} of the ideal
$I$.  We will assume that $a_{1}\leq\cdots\leq a_{t}$ and $b_{1}\leq\cdots\leq
b_{t+c-1}$, so the entries of $A$ increase from left to right and from the bottom to the top.
Since $a_{i,i} \le 0$ implies that all the minors containing the first $i$ columns are zero, we can assume without loss of generality  that $a_{i,i}>0$ for all $i$.
%It is not difficult to see that the entries on the diagonal of $A$
%have to be positive i.e. , since otherwise the first minor of $M$ would be zero (see e.g. \cite[Proposition 2.4]{Go04}).\\

For the Hilbert series of a standard graded $\Bbbk$-algebra $S/I$ we will use the notation $\HS_{S/I}$.
We will write the Hilbert series in rational form as
$$\HS_{S/I}(z)={\displaystyle
\frac{\HP(z)}{(1-z)^{d}}},$$ where $d$ is the Krull dimension of $S/I$. The numerator
$\HP(z)=1+h_{1}z+h_{2}z^{2}+\cdots+h_{s}z^{s}$, with $h_s\neq 0$,  is called the \emph{$h$-polynomial }of $S/I$
and  its coefficients form  the \emph{$h$-vector} of $S/I$, $h_{S/I}=(h_0,h_1,\ldots,h_s)$.
The degree matrix $A$ of $I$ determines the minimal free resolution of $S/I$ (given by the Eagon-Northcott
complex) and therefore also $h_{S/I}$. As in this paper we study the $h$-vectors of standard determinantal ideals,  we will write $h^A$ and $\HP^A(z)$
instead of $h_{S/I_{max}(M)}$, respectively $\HP_{S/I_{max}(M)}(z)$. We denote  by $\tau(h^A)$ the degree of the $h$-polynomial.

Our key lemma is based on liaison theory.  We recall here briefly the notion of  basic double link. If
$\mfb\subseteq \mfa\subseteq S$ are two homogeneous ideals such
that $\mfb$ is Cohen-Macaulay, $ht(\mfa)=ht(\mfb)+1$, and $f\in 
NZD_S(S/\mfb)$ is a homogeneous non-zero-divisor, then the ideal
$I=f\cdot\mfa+\mfb$ is called a \emph{basic double link} 
of $\mfa$. 
The terminology is motivated by Gorenstein liaison theory: In the above notation,
$I$ can be Gorenstein linked to $\mfa$ in two steps if $\mfa$ is unmixed, and $S/\mfb$ is Cohen-Macaulay and generically Gorenstein (see  \cite[Proposition 5.6]{KMMNP01} and \cite[Theorem
3.5]{Ha07}).
 In \cite{Go07},  Gorla constructed basic double links in which all the ideals involved are standard determinantal (see also \cite{KMMNP01} for more general results in this direction). 
We will use this construction to prove the following recursive formula for the $h$-vector of a standard determinantal ideal.

For any matrix $A$ and positive integers $k$ and $l$ we use the following notation: 
$A^{(k,l)}$ is the matrix obtained from $A$  by deleting the $k$-th row
and $l$-th column. By convention, $A^{(k,0)}$ (respectively $A^{(0,l)}$) means that
only
the $k$-th row (respectively the $l$-th column) has been deleted.

\begin{lemma}\label{key}
Let $A=\left(a_{i,j}\right)\in\mathbb{Z}^{t\times (t+c-1)}$ be a degree matrix.
For any $k=1,\ldots,t$ and\\ $l=1,\ldots,t+c-1$, such that
$a_{k,l}\geq 0$, we have 
$$\HP^{A}(z)=z^{a_{k,l}}\HP^{A^{(k,l)}}(z)+(1+\cdots+z^{a_{k,l}-1})\HP^{A^{(0,l)} } (z).$$
\end{lemma}

\begin{proof}  We distinguish two cases:\\
\emph{Case 1:} $a_{k,l}>0$.
Without loss of generality we can assume that $k,l=1$. Consider the homogeneous matrix 
\begin{center}
$M=\begin{pmatrix}f_{1,1} & f_{1,2} & \cdots &  & f_{1,t+c-1}\\
0 & f_{2,2} & \cdots &  & f_{2,t+c-1}\\
\\
\vdots & \vdots &  &  & \vdots\\
\\
\\
0 & f_{t,2} & \cdots &  & f_{t,t+c-1}
\end{pmatrix}$,
\end{center}
where the $f_{i,j}$'s are generically chosen forms in $S=\Bbbk[x_0,\dots,x_n]$, with $n\ge c-1$ and $\deg(f_{i,j})=a_{i,j}$. Such forms exist when because the field $\Bbbk$ is infinite. 
Let  $\mfa=I_{\max}(M^{(1,1)})$ and $\mfb=I_{\max}(M^{(0,1)})$, be two ideals which by the generic choice of the forms $f_{i,j}$ are standard determinantal.
Thus, by construction, we have $ht(\mfb)=ht(\mfa)-1$ and 
$f_{1,1}$ is a non-zero-divisor in $S/\mfb$.
If $I=I_{\max}(M)$, then by direct computation on the generators we obtain that 
$$I=f_{1,1}\mfa+\mfb,$$ 
so $I$ is a basic double link of $\mfa$. 
By \cite[Theorem 3.1]{Go07},  the ideal $I$ is also standard determinantal. 
Notice that the corresponding degree matrices of $I$, $\mfa$ and $\mfb$ are $A$, $A^{(1,1)}$, respectively $A^{(0,1)}$.
From the short exact sequence
$$\xymatrix{0\ar[r]&\mfb(-a_{1,1})\ar[r]&\mfa(-a_{1,1})\oplus \mfb\ar[r]&I\ar[r]&0,}$$
where the first map is given by  $g\mapsto(g,f_{1,1}\cdot g)$ and
the second by
$(g,h)\mapsto gf_{1,1}-h$, it follows that, if $d=n-c+1$, then
\begin{align}\HS_{S/I}(z)= \frac{\HP^A}{(1-z)^{d}}
&=z^{a_{1,1}}\HS_{S/\mfa}(z)+(1-z^{a_{1,1}})\HS_{S/\mathfrak { b } } (z)\notag\\
 &={\displaystyle \frac{z^{a_{1,1}}\HP^{A^{(1,1)}}(z)}{(1-z)^{d}}+{\displaystyle \frac{(1-z^{a_{1,1}})\HP^{A^{(0,1)}}(z)}{(1-z)^{d+1}}}}\notag\\
  &=\displaystyle{\frac{z^{a_{1,1}}\HP^{A^{(1,1)}}(z)+(1+\cdots+z^{a_{1,1}-1})\HP^ {A^{(0,1)}}(z)}{ (1-z)^{d} } }\notag \end{align}
 and we conclude.\\

\noindent\emph{Case 2:} $a_{k,l}=0$. By induction on $t$ and $c$ we will show that
%\begin{center}
 $\HP^A(z)=\HP^{A^{(k,l)}}(z)$.
%\end{center}
%\noindent For $t=1$, $A=\left(a_{1,1},\ldots,a_{1,c}\right)$ and there is nothing to show, as $a_{1,i}>0\,\forall i$.\\
By the ordering of the entries in $A$, and because $a_{i,i}>0$ for all $i$, if  $a_{k,l}=0$, then  $k>l$ (i.e.  $a_{k,l}$ lies below the diagonal).

When $c=1$, the $h$-vector corresponding to $A$ is just a sequence of 1's of length
$\tr(A)={ \sum_{i=1}^{t}a_{i,i}}$. %So we only need to show that $\tr(A)=\tr(A^{(k,l)})$.
Notice that
$$\tr(A^{(k,l)})={\displaystyle \sum_{i=1}^{l-1}a_{i,i}}+{\displaystyle
\sum_{i=l}^{k-1}a_{i,i+1}}+{\displaystyle {\displaystyle
\sum_{i=k+1}^{t}a_{i,i}}}.$$
By the homogeneity of $A$, we have $\tr(A^{(k,l)})=\tr(A^{(k,l)})+a_{k,l}=\tr(A)$.

The first row has only positive entries, so $t\ge2$.
For $t=2$, since $a_{2,1}=0$,  from \emph{Case 1} applied to the indices $(2,c+1)$, it
follows that
\begin{equation}\label{t=2}
\HP^A(z) = z^{a_{2,c+1}}\HP^{\left(a_{1,1},\dots,a_{1,c}\right)}(z) + (1+\dots+z^{a_{2,c+1}-1})\HP^{A^{(0,c+1)}}(z).
\end{equation}
The $h$-polynomial of a 1-row degree matrix is the $h$-polynomial of the corresponding complete intersection, namely:
$$\HP^{\left(a_{1,1},\dots,a_{1,c}\right)}(z)= \prod_{i=1}^{c}\left(1+\dots+z^{a_{1,i}-1}\right).$$
By induction on $c$ we have $\HP^{A^{(0,c+1)}}(z) = \HP^{\left(a_{1,2},\dots,a_{1,c}\right)}(z)$, so \eqref{t=2} becomes
\begin{eqnarray*}
\HP^A(z) &=& z^{a_{2,c+1}} \prod_{i=1}^{c}\left(1+\dots+z^{a_{1,i}-1}\right) + \left(1+\dots+z^{a_{2,c+1}-1}\right) \prod_{i=2}^{c}\left(1+\dots+z^{a_{1,i}-1}\right)\\
&= &\left(1+\dots + z^{a_{1,1}+a_{2,c+1}-1}\right)\prod_{i=2}^{c}\left(1+\dots+z^{a_{1,i}-1}\right).
\end{eqnarray*}
As $A$ corresponds to the degrees in a homogeneous matrix, we have $a_{1,1}+a_{2,c+1} = a_{2,1}+a_{1,c+1}$ and we conclude.
%
%\begin{center}
% $h^A_i=h^{(a_{1,1},\ldots,a_{1,c})}_{i-a_{2,c+1}}+\displaystyle\sum_{k=0}^{a_
%{2,c+1}-1}h_{i-k}^{(a_{1,2},\ldots,a_{1,c})} $
%\end{center}
%
%\noindent Since the $h$-vector of a complete intersection of type
%$(a_{1,1},\ldots,a_{1,c})$ can also be written in the form 
%\begin{center}
% $h^{(a_{1,1},\ldots,a_{1,c})}_i=\displaystyle\sum_{k=0}^{a_{1,1}-1}h_{i-k}^{
%(a_{1,2},\ldots,a_{1,c}) } $
%\end{center}
%
%\noindent We have that
%\begin{align}
%h^A_i&=\displaystyle\sum_{k=a_{2,c+1}}^{a_{1,1}+a_{2,c+1}-1}h_{i-k}^{
%(a_{1,2},\ldots,a_{1,c}) }+\displaystyle\sum_{k=0}^{a_{2,c+1}-1}h_{i-k}^{
%(a_{1,2},\ldots,a_{1,c}) }\notag\\
%&=\displaystyle\sum_{k=0}^{a_{1,c+1}-1}h_{i-k}^{
%(a_{1,2},\ldots,a_{1,c}) }=h_i^{(a_{1,2},\ldots,a_{1,c+1})}\notag
%\end{align}
%
%

When $t>2$,  there exists some positive entry $a_{i,i}$, with $i\neq k,l$.
The matrices $A^{(i,i)}$ and $A^{(0,i)}$ contain $a_{k,l} = 0$. Applying \emph{Case 1} for $a_{i,i}$, and using the induction on $t$ and $c$ we obtain
\begin{eqnarray*}
\HP^{A}(z)&=&z^{a_{i,i}}\HP^{A^{(i,i)}}(z)+\left(1+\cdots+z^{a_{i,i}-1}\right)\HP^{A^{(0,i)}}(z)\\
&=&z^{a_{i,i}}\HP^{(A^{(i,i)})^{(k,l)}}(z)+\left(1+\cdots+z^{a_{i,i}-1}\right)\HP^{(A^{(0,i)})^{(k,l)}}(z)\\
&=&\HP^{A^{(k,l)}}(z)
\end{eqnarray*}
\end{proof}

\begin{remark}\label{rem:recursive}  Lemma \ref{key}
implies the following recursive formula for the $h$-vector of $A$:
$$h^{A}_i=h^{A^{(k,l)}}_{i-a_{k,l}}+\displaystyle{\sum_{k=0}^{a_{k,l}-1}
h^{A^{(0,l)}}_{i-k}}.$$
In particular, if  some entry $a_{k,l}=0$, then $h^A=h^{A^{(k,l)}}$.
 As we are interested in studying the h-vectors of standard
determinantal ideals,  we may assume from now on that none of the degree matrices contain zeros.
\end{remark}

\section{Formulae}

In this section we find a general formula for the $h$-polynomial in terms of the entries of the degree matrix (Proposition \ref{prop:h-poly}). We then compute the length and the last entry of the $h$-vector (Lemma \ref{lem:length}). Finally, we give a more explicit description of the last entries of the $h$-vector when all the rows in the degree matrix are equal (Proposition \ref{prop:lastentries}). 

%For any degree matrix  $A=\left(a_{i,j}\right)\in
%\mathbb{Z}^{t\times t+c-1}$ we define the vectors $d=(d_1,\ldots,d_{t+c-1})$ and
%$e=(e_1,\ldots,e_t)$ as follows: $d_i=a_{1,i}-a_{1,1}$ and
%$e_i=a_{1,1}-a_{i,1}$. Notice that  $d_1\leq d_2\leq \cdots\leq
%d_{t+c-1}$, $e_1\leq e_2\leq \cdots \leq e_t$ and
%$a_{i,j}=a_{1,1}+d_j-e_i$.
%%%%%%%%%%%%%%%%%%%%%%%%%%%%%%%%%%%%
Recall that the $h$-polynomial of a complete intersection generated in degrees $(d_1,\dots,d_c)$ is
\begin{equation}\label{eq:ci}
 \HP^{(d_1,\dots,d_c)}(z) = \prod_{i=1}^{c}\left(1+z+\dots+z^{d_i-1}\right).
 \end{equation}
We fix the following notation.
Let $a,b\ge 0$ be two integers. For any increasing sequence of integers $0<i_1<\dots<i_{b}<a+b$ and any matrix $A=\left(a_{i,j}\right)\in\ZZ^{a\times (a+b)}$,   we define two ordered sets of integers:
\begin{eqnarray*}
\{j_1,\dots,j_{i_b-b}\} &=& \{1,\dots,i_{b}\}\setminus\{i_1,\dots,i_{b}\}\\
\textup{g}_A(i_1,\ldots,i_{b}) &=& \{a_{i_1,i_1},\,\,a_{i_2-1,i_2},\,\,\ldots,\,\, a_{i_{b}-(b-1),i_
{b}},\,\,\sum_{i=i_{b}-(b-1)}^{a} a_{i,i+b}\}\\
\end{eqnarray*}
To the first set we associate a nonnegative integer; to the second set a polynomial in one variable:
\begin{eqnarray*}
e_A(i_1,\ldots,i_{b}) &=& \sum_{i=1}^{i_{b}-b} a_{i,j_i}\\
\hci_A(i_1,\dots,i_{b}) &= &
\HP^{(\textup{g}_A(i_1,\ldots,i_{b}))}(z) 
\end{eqnarray*}
\begin{proposition}\label{prop:h-poly}
The $h$-polynomial of any degree matrix $A=\left(a_{i,j}\right)\in\mathbb{Z}^{t\times (t+c-1)}$ is
given by
$$\HP^{A}(z) = \sum_{0<i_{1}<\cdots<i_{c-1}<t+c-1} z^{e_A(i_1,\ldots,i_{c-1})}
\cdot \hci_A(i_1,\dots,i_{c-1}). $$ 
\end{proposition}
\begin{proof}
For $t=1$ and  any $c\ge 1$ we obtain only one summand, and the equality clearly holds. We will use induction on $t$ and on $c$. The recursive formula of Lemma \ref{key} gives us
$$\HP^{A}(z)=z^{a_{1,1}}\HP^{A^{(1,1)}}(z)+(1+\cdots+z^{a_{1,1}-1})\HP^{A^{(0,1)
} } (z).$$
Let us denote  the entries of the matrix $A^{(1,1)}$ by $(a'_{i,j})$  and the entries of $A^{(0,1)}$ by $(a''_{i,j})$. By definition  $a'_{i,j}= a_{i+1,j+1}$ and $a''_{i,j}=a_{i,j+1}$.
By the inductive hypothesis on $t$ we have
$$\HP^{A^{(1,1)}}(z)= \sum_{0<i_{1}<\cdots<i_{c-1}<t+c-2} z^{e_{A^{(1,1)}}(i_1,\ldots,i_{c-1})} \cdot \hci_{A^{(1,1)}}(i_1,\dots,i_{c-1}).$$
For any sequence ${0<i_{1}<\cdots<i_{c-1}<t+c-2}$ we have
$$e_{A^{(1,1)}}(i_1,\ldots,i_{c-1}) = \sum_{i=1}^{i_{c-1}-(c-1)} a'_{i,j_i} = \sum_{i=2}^{i_{c-1}+1-(c-1)} a_{i,j_i} =  e_A(i_1+1,\dots,i_{c-1}+1) -a_{1,1}.$$
It is easy to check that that this implies 
$$\HP^{A^{(1,1)}}(z)= \sum_{1<i_{1}<\cdots<i_{c-1}<t+c-1} z^{e_{A}(i_1,\ldots,i_{c-1}) - a_{1,1}} \cdot \hci_{A}(i_1,\dots,i_{c-1}).$$
By the inductive hypothesis on $c$ we obtain
$$\HP^{A^{(0,1)}}(z)= \sum_{0<i_{1}<\cdots<i_{c-2}<t+c-2} z^{e_{A^{(0,1)}}(i_1,\ldots,i_{c-2})} \cdot \hci_{A^{(0,1)}}(i_1,\dots,i_{c-2}).$$
It is easy to check as above that $ \textup{g}_{A^{(0,1)}}(i_1,\ldots,i_{c-2}) = \textup{g}_A(1,i_1+1,\ldots,i_{c-1}+1) \setminus\{a_{1,1}\}$, and that $e_{A^{(0,1)}}(i_1,\ldots,i_{c-2}) = e_A(1,i_1+1,\dots,i_{c-2}+1)$. This implies that 
$$\HP^{A^{(0,1)}}(z)= \sum_{1<i_{2}<\cdots<i_{c-1}<t+c-1}
z^{e_{A}(1,i_2,\ldots,i_{c-1})} \cdot
\frac{\hci_{A}(1,i_2,\dots,i_{c-1})}{1+\dots +z^{a_{1,1}-1}},$$
and by Lemma \ref{key} we conclude.

\end{proof}

%%%%%%%%%%%%%%%%%%%%%%%%%%%%%%%%%%%%%

We now focus on the degree and the leading coefficient of the $h$-polynomial. 

\begin{lemma}\label{lem:length}
Let $A=\left(a_{i,j}\right)\in \mathbb{Z}^{t\times (t+c-1)}$ be a degree matrix and
let $h^A=(h_0,\ldots,h_{\tau(h^A)})$. Then:
\begin{enumerate}
 \item[\textup{(i)}]$\tau(h^A)=a_{1,1}+\cdots+a_{1,c}+a_{2,c+1}+\cdots+a_{t,t+c-1}-c$.
 \item[\textup{(ii)}] $h_{\tau(h^A)}=\displaystyle\binom{r+c-2}{c-1}$, where
$r=\max\{i ~:~\,a_{1,1}=\cdots=a_{i,1}\}$.
\end{enumerate}
\end{lemma}

\begin{proof} We will prove the claim by induction on $t$ and $c$.
For $t,c=1$, both statements are clear, so  let $t,c>1$. Comparing the degrees of the $h$-polynomials in Lemma \ref{key}, with $(k,l) = (t,t+c-1)$ we obtain
\begin{equation}\label{eq:deg}
\tau(h^A)=\max\{\tau(h^{A^{(t,t+c-1)}})+a_{t,t+c-1},\tau(h^{A^{(0,t+c-1)
}})+a_{t,t+c-1}-1\}
\end{equation}
and (i) follows by induction.

From \eqref{eq:deg} and (i) we deduce that, if  $a_{t,t+c-2} < a_{t-1,t+c-1}$, then the leading coefficients of $\HP^A(z)$ and $\HP^{A^{(t,t+c-1)}}(z)$ are equal.
%$h_{\tau\left(h^A\right)} = h_{\tau\left(h^{A^{(t,t+c-1)}}\right)}$.
Thus it is enough to prove the second statement for matrices with equal rows (i.e. with $r=t$).
If we denote by $A' = A^{(t,t+c-1)}$ and $A'' = A^{(0,t+c-1)}$, and apply  Lemma \ref{key} for  $(k,l) = (t,t+c-1)$ we obtain
$$h_{\tau\left(h^A\right)}=h^{A'}_{\tau\left(h^{A'}\right)}+h^{A''}_{\tau\left(h^{A''}\right)}
=\displaystyle\binom{t+c-3}{c-1}+\binom{t+c-3}{c-2}=\binom{t+c-2}{c-1}$$

\end{proof}

%In the following, for any integer sequence $h=(h_0,\ldots,h_s)$  and any
%integer $a$, we will write $h(a)$ for the vector
%$(\underbrace{0,\ldots,0}_{a},h_0,\ldots,h_s)$.
From now on, $r$ will denote the \emph{number of maximal equal rows} in a degree matrix. That is $$r=\max\{i ~:~\,a_{1,1}=\cdots=a_{i,1}\}$$

\begin{remark}\label{rem:compute}
 Let $A\in \mathbb{Z}^{t\times (t+c-1)}$ be a degree matrix,
 %without zero entries
and let $h^A=(h_0,\ldots,h_s)$. We denote by $h' = (h'_0,\dots,h'_{s'})$ the $h$-vector of $A^{(t,t+c-1)}$ and by $h'' = (h''_0,\dots,h''_{s''})$ the
vector given by
$$h''_i=\sum_{k=0}^{a_{t,t+c-1 }-1 } h_{i-k}^{A^{(0,t+c-1)}},$$
where $h_{i-k}^{A^{(0,t+c-1)}} = 0$ if $i<k$.
Lemma \ref{key} states that $h^A$ is computed by component-wise addition: 
$$
\begin{array}{lllllllllll}
 0 & \ldots & 0 & h'_{0} & \ldots & h'_{s'-a_{1,1}+a_{t,1}} & h'_{s'-a_{1,1}+a_{t,1}+1} & \ldots & h'_{s'}& +\\
\rule[-2ex]{0pt}{5ex} h''_{0} & \ldots & h''_{a_{t,t+c-1}-1}& h''_{a_{t,t+c-1}} & \ldots &  h''_{\tau(h'')} &0 &  \ldots& 0&\\
\hline \rule{0pt}{3ex}  h_{0}^{A} & \ldots & h^A_{a_{t,t+c-1}-1} & h^A_{a_{t,t+c-1}} & \ldots &   h_{s-a_{1,1}+a_{t,1}}^{A} &h_{s-a_{1,1}+a_{t,1}+1}^{A}&\ldots & h_{s}^{A}&
\end{array}
$$
By Lemma \ref{lem:length} we have  $s'-s''=a_{1,1}-a_{t,1}$. 
In particular, as $a_{1,1}=a_{r,1}>a_{r+1,1}\ge  \dots \ge a_{t,1}$, the last $a_{1,1}-a_{r+1,1}$ entries of $h^A$ are equal to the last $a_{1,1}-a_{r+1,1}$ entries of $h^{\bar{A}}$, where $\bar{A}$ is the $r\times(r+c-1)$ upper-left block of $A$.
%$e_i=a_{1,1}-a_{i,1}$
%Notice that in the special case when $A$ has $r=t-1$ equal maximal rows, using
%the fact that $A$ is homogeneous and does not have zero entries, one can obtain
%a range for $e_t$, namely:
%
%\begin{center}
% $e_t=e_{r+1}\in \{1,\ldots,a_{1,r+1}\}\setminus\{a_{1,1},\ldots,a_{1,r}\}$
%\end{center}
\end{remark}

The following proposition  describes the last part of the $h$-vector of a
degree matrix with equal rows. By the above remark these values provide lower bounds for the last entries of the $h$-vector of any degree matrix.   In what follows, we use the convention that
$ \binom{a}{b}=0$,  if $b<0$ or $a<b$.

\begin{proposition}\label{prop:lastentries}
 Let $A\in \mathbb{Z}^{r\times (r+c-1)}$ be a degree matrix with equal rows. Denote by $s=\tau(h^A)$ and by
$a_j=a_{l,j},~\forall~ l,j$. For any $i=0,\ldots,a_{r+1}-1$ we have:
\begin{align}
 h^{A}_{s-i}=&\displaystyle{\binom{r+c-2}{c-1}\cdot\binom{c+i-1}{c-1}}+
\notag\\&+\displaystyle{\sum_{\aa=1}^{c-1}(-1)^{\aa}\binom{r-\aa+c-2}{c-1-\aa}
\sum
_{1\leq j_1<\cdots<j_{\aa}\leq
r}\binom{c+i-1-a_{j_1}-\cdots-a_{j_{\aa}}}{c-1}}.\notag
\end{align}

\end{proposition}
\begin{proof}We will prove the claim by induction on $r$ and $c$ using
the binomial formula 
\begin{equation}\label{eq:binom}
\sum_{i=0}^{a-1}\binom{d-i}{b}=\binom{d+1}{b+1}-\binom{d-a+1}{b+1}.
\end{equation}
The case $c=1$ corresponds to a hypersurface, and  the claim  clearly holds. When $r=1$,  denote by
$h'=(h'_0,\ldots,h'_{s'})$ the $h$-vector of a complete intersection of type
$(a_2,\ldots,a_c)$. For $i=0,\ldots,a_2-1$, using \eqref{eq:ci}, induction on $c$, and \eqref{eq:binom} we obtain:
\begin{align}
 h_{s-i}&=\displaystyle{\sum_{k=0}^{a_1-1}h'_{s'-(i-k)}}
\notag\\
&=\displaystyle{\sum_{k=0}^{a_1-1}\binom{c-2+i-k}{c-2}}+\displaystyle{\sum_{k=0}
^{a_1-1}\binom{c-2+i-k-a_2}{c-2}}\notag\\
&=\binom{c-1+i}{c-2}-\binom{c-1+i-a_1}{c-1}.\notag
\end{align}

Let now  $r,c>1$. We will write for shortness 
$h^{A^{(1,1)}}= (h'_0,\dots,h'_{s'})$ and $h^{A^{(0,1)}}=(h''_0,\dots,h''_{s''})$,
By Lemma \ref{lem:length} we have $s=s'+a_1$ and $s=s''+(a_1-1)$.  By Remark \ref{rem:recursive}  we have in this notation that, for any $i=0,\ldots,a_{r+1}-1$,
\begin{equation}\label{eq:rec}
h^{A}_{s-i}=h'_{s'-i}+\sum_{j=0}^{a_1-1}h''_{s''-(i-j)}.
\end{equation}
For the following computation we use  induction on $c$ and $r$, the formula \eqref{eq:binom}, and  the correspondence
between the indices in $A,A'$ and $A''$. We  also take  into account that, if
$j_\aa=a_{r+1}$, then for any $i=0,\ldots,a_{r+1}-1$  we have 
$\binom{c+i-1-a_{j_1}-\cdots-a_{j_\aa}}{c-1}=0$.

\begin{eqnarray*}
h'_{s'-i}&=&\binom{r+c-3}{c-1}\binom{c-1+i}{c-1}+\\
&&+\sum_{\aa=1}^{c-1}(-1)^{\aa}\binom{r-\aa+c-3}{c-1-\aa}\sum_{2\leq j_1<\cdots<j_{\aa}\leq r}\binom{c+i-1-a_{j_1}-\cdots-a_{j_{\aa}}}{c-1}.\\
&&\\
\sum_{j=0}^{a_1-1}h''_{s''-(i-j)}&=&\binom{r+c-3}{c-2}\binom { c-1+i } { c-1}-\binom{r+c-3}{c-2}\binom { c-1+i-a_1 } { c-1}+\\
 &&+\sum_{\aa=1}^{c-1}(-1)^{\aa}\binom{r-\aa+c-3}{c-2-\aa}\sum_{2\leq j_1<\cdots<j_{\aa}\leq r}\binom{c+i-1-a_{j_1}-\cdots-a_{j_{\aa}}}{c-1}-\\
&&-\sum_{\aa=1}^{c-2}(-1)^{\aa}\binom{r-\aa+c-3}{c-2-\aa}\sum_{2\leq j_1<\cdots<j_{\aa}\leq r}\binom{c+i-1-a_1-a_{j_1}-\cdots-a_{j_{\aa}}}{c-1}.
 \end{eqnarray*}

Substituting these formulae in \eqref{eq:rec}, grouping the summands, and applying \eqref{eq:binom} we obtain

\begin{eqnarray*}
 h^A_{s-i}&=&\binom{r+c-2}{c-1}\binom{c-1+i}{c-1}\\
&&+\sum_{\aa=1}^{c-1}(-1)^{\aa}\binom{r-\aa+c-2}{c-1-\aa}\sum
_{2\leq j_1<\cdots<j_{\aa}\leq r}\binom{c+i-1-a_{j_1}-\cdots-a_{j_{\aa}}}{c-1}+\\
&&+\sum_{\aa=2}^{c-1}(-1)^{\aa}\binom{r-\aa+c-2}{c-1-\aa}\sum_{2\leq j_1<\cdots<j_{\aa-1}\leq r}\binom{c+i-1-a_1-a_{j_1}-\cdots-a_{j_{\aa-1}}}{c-1}-\\
&&-\binom{r+c-3}{c-2}\binom { c-1+i-a_1 } { c-1},
 \end{eqnarray*}

% \begin{align}
 % h^A_{s-i}&=\displaystyle{\binom{r+c-2}{c-1}\binom{c-1+i}{c-1}}\notag\\
  %&+\displaystyle{\sum_{\aa=1}^{c-1}(-1)^{\aa}\binom{r-\aa+c-2}{c-1-\aa}\sum
%_{2\leq j_1<\cdots<j_{\aa}\leq
%r}\binom{c+i-1-a_{j_1}-\cdots-a_{j_{\aa}}}{c-1}}\notag\\
%+&\displaystyle{\sum_{\aa=1}^{c-1}(-1)^{\aa}\binom{r-\aa+c-2}{c-1-\aa}\sum
%_{2\leq j_1<\cdots<j_{\aa-1}\leq
%r}\binom{c+i-1-a_1-a_{j_1}-\cdots-a_{j_{\aa-1}}}{c-1}}\notag
% \end{align}
and the claim follows by straight forward  rewriting of this formula.
\end{proof}
\newpage
\section{Standard determinantal ideals and pure O-sequences}
In this section we prove the main result of this paper (Theorem \ref{thm:main}), which states that if all the rows in a degree matrix are equal, then its $h$-vector is a log-concave pure O-sequence. By Proposition \ref{prop:level} such matrices correspond exactly to level standard determinantal ideals.
We  conjecture the converse of the main theorem to hold (Conjecture \ref{conj:main}). 
Among the support we bring for this statement are the validity for codimension two, for the last entry of the $h$-vector equal to one, and for matrices with all entries positive.

 In codimension one, all $h$-vectors are finite  sequences of 1s, thus pure O-sequences. We will assume throughout this section that the codimension $c$ is greater than two. \\

An  \emph{O-sequence} is a finite vector of integers which is the Hilbert function of some standard graded artinian algebra, that is it satisfies the numerical conditions in Macaulay's theorem \cite{Mac27}. An O-sequence is called \emph{pure} if it is the Hilbert function of a level, monomial artinian algebra. In general, a Cohen-Macaulay, standard graded quotient of the polynomial ring $S$ is called \emph{level} if the last $S$-module in its minimal free resolution is of the form $S(-s)^a$, where $s$ and $a$ are positive integers. 

 Pure O-sequences have also a purely combinatorial interpretation as follows.
We will write $\Mon(\mathbf{y})$ for the collection of all monomials in $\mathbf{y} = (y_1,\dots,y_m)$. An \emph{order ideal} of $\Mon(\mathbf{y})$ is a finite subset $\G\subset \Mon(S)$ closed under division, i.e. if $N|M\in \G$,  then $N\in \G$. 
%The partial order given by the divisibility of monomials gives $\G$ a poset structure. 
An order ideal is called \emph{pure} if all maximal monomials  have the same degree. We write 
$$\G=\langle M \in \G ~:~M \textup{~is maximal with respect to division}\rangle.$$
The  \emph{$f$-vector} of an order ideal $\G$ is $f(\G)=(f_0,\ldots,f_s)$, 
where $f_i(\G)=|\{M\in \G|\deg(M)=i\}|$. It is not difficult to check that a vector $h= (h_0,\dots,h_s)$ is a pure O-sequence if and only if it is the $f$-vector of some pure order ideal. Recall that the vector $h$ is called \emph{log-concave}, if 
$$h_i^2\geq h_{i-1}\cdot h_{i+1},~\forall~ i=1,\ldots,s-1.$$ 
 
We use matroids to obtain a connection between $h$-vectors of standard determinantal ideal and pure O-sequences. 
A  \emph{simplicial complex} on the vertex set $[n] = \{1,\dots,n\}$ is a collection of subsets $\D\subseteq2^{[n]}$ closed under taking subsets, i.e. if $G\subseteq F \in \D$, then $G\in\D$. A \emph{matroid} is a simplicial complex with the extra property that  if $F,G\in \D$,  with $|G|<|F|$, then there exists $i\in F$ such that $G\cup \{i\}\in \D$.
%For a positive integer $n$ we will write $[n]$ for $\{1,\ldots,n\}$. A simplicial complex $\D$ on $[n]$ is a collection of subsets on $[n]$ with the property that for any $F,F’\in \D$, such that $F’\subset F$, holds $F’\in \D$. An element $F\in \D$ is called a face of $\D$. The faces of $\D$ which are maximal with respect to inclusion will be called facets.
% We use the notation
%\begin{center}
%$\mathcal{F}(\D)=\{F\in \D|\,F\,\text{is a facet}\}$.
%\end{center}
For  $v\in \D$, the \emph{link of $v$ in $\D$}, respectively the \emph{deletion of $v$ in $\D$} are the  simplicial complexes
\begin{eqnarray*}
\link_{\D}(v)&=&\{F\in \D~;~v\notin F\,\text{and}\,F\cup \{v\}\in
\D\},\\
\D\setminus v&=&\{F\in \D~;~v\notin F\}.
\end{eqnarray*}
When $\D$ is a matroid, then both $\link_{\D}(v)$ and $\D\setminus v$ are matroids as well. 
The maximal faces under inclusion are called \emph{facets}; they determine the simplicial complex. We denote the set of facets of $\D$ by $\mathcal{F}(\D)$.
A vertex $v\in \D$ with $v\in F $ for any $F\in \mathcal{F}(\D)$  is
called a \emph{cone point} of $\D$.

For any simplicial complex $\D$, the  \emph{cover ideal} is the square-free monomial ideal of the polynomial ring $S=\Bbbk[x_1,\dots,x_n]$ defined as
 $$J(\D)=\bigcap_{F\in\D}(x_i ~:~\,i\in F) .$$
We will denote by  $h^\D$  the $h$-vector of $S/J(\D)$.
According to \cite[Remark 1.7]{CV3}, if $\D$ is a matroid and $v\in \D$ not a cone point, then 
\begin{equation}\label{eq:recursive for matroids}
 h^{\D}_i=h^{\D\setminus v}_{i-1}+h^{\link_{\D}(v)}_i.
\end{equation}

\begin{remark}\label{rem:matroid h}
For any simplicial complex $\D$, the \emph{dual} (or complement) of $\D$  is the simplicial  complex $\D^\opc$ with 
$$\mathcal{F}(\D^\opc) = \left\{ [n]\setminus F~:~ F\in\mathcal{F}(\D)\right\}.$$ 
A classical matroid theory result states that $\D$ is a matroid if and only if $\D^\opc$ is a matroid \cite{Ox11}.

In common matroid terminology, the vector $h^\D$  we defined above is the \lq\lq classical\rq\rq~ $h$-vector of the dual matroid. This choice was made in order to keep a coherent notation with the main result of \cite{CV3}, which we use to prove the following theorem.
\end{remark}

\begin{theorem}\label{thm:main} Let $X\subseteq\mathbb{P}^{n}$ be a codimension c standard
determinantal scheme. If the corresponding degree matrix $A$  has equal rows, then
$h^A$ is a log-concave pure O-sequence. 

In particular,  if  $A\in \mathbb{Z}^{t\times(t+c-1)}$, with rows  $(a_1,\dots,a_{t+c-1})$,
then  $h^A=f(\G)$ where $\G$ is the order ideal of $\Mon(y_1,\ldots,y_c)$ given by
$$\G=\left\langle 
y_{1}^{(\sum_{i=1}^{l_{1}-1}a_{i})-1}\cdot
y_{2}^{(\sum_{i=l_{1}}^{l_{2}-1}a_{i})-1}\cdots
y_{c}^{(\sum_{i=l_{c-1}}^{t+c-1}a_{i})-1}~:~\forall1=l_{0}<l_{1}<\cdots<l_{c-1}
\leq t+c-1 \right\rangle. $$
\end{theorem}

\begin{proof}
We will write  $m=t+c-1$ for short. For $i=1,\ldots,m$, let $A_{i}$
be a set of vertices of cardinality $a_i$. As in \cite{CV3}, we define the simplicial complex $\D_0(c,m,(a_{1},\ldots,a_{m}))$ on $\sqcup_{i=1}^mA_i$ as
$$\langle \{v_{i_{1}},\ldots,v_{i_{c}}\} ~:~1\leq i_{1}<\cdots<i_{c}\leq m \textup{~and~} v_{i_j}\in
A_{i_j}\textup{~for every~} i_j\rangle.$$
One can easily check that $\D_0(c,m,(a_{1},\ldots,a_{m}))$ is a matroid. We will
show by induction on $c$
and $t$ that the $h$-vectors $h^A$ and
$h^{\D_0(c,m,(a_{1},\ldots,a_{m}))}$  coincide. For $t=1$ or $c=1$ the claim  is straight forward. 
Let $t,c>1$. By Lemma \ref{key} applied for $a_m$ we have
$$h=h^{A^{(t,m)}}_{i-a_{m}} + \sum_{k=0}^{a_{m}-1}h^{A^{(0,m)}}_{i-k}.$$
%where $h'$ and $h''$ are the h-vectors corresponding to the degree matrices $A'=A^{(t,m)}$ and $A''=A^{(0,m)}$ respectively. 
On the other hand, applying $a_m$ times the formula \eqref{eq:recursive for matroids}, once for every vertex in $A_m$, we obtain
 $$h^{\D_0(c,m,(a_{1},\ldots,a_{m}))}_i=h^{\D_0(c,m-1,(a_{1},\ldots,a_{m-1}))}_{i-a_m}+\sum_{k=0}^{ a_m-1}h_{i-k}^{\D_0(c-1,m-1,(a_{1},\ldots,a_{m-1}))},$$
and we conclude by induction. In particular, by \cite[Theorem 3.5]{CV3}, $h^A$ is the pure O-sequence given by $\G$  as claimed.

Furthermore $\D_0(c,m,(a_{1},\ldots,a_{m}))$ is representable over any infinite  field $\mathbb{F}$ of characteristic zero. A presentation matrix $D$ can be constructed as follows: choose $m$ generic vectors $w_{1},\ldots,w_{m}\in\mathbb{F}^{c}$, that is any $c$ of them are linearly independent. Let the first $a_1$ columns of $D$ be $w_1$, the next $a_2$ be equal to $w_2$ and so on. Clearly $D$ represents the matroid $\D_0(c,m,(a_{1},\ldots,a_{m}))$.  
 A matroid is representable over $\mathbb{F}$ if and only if its dual is representable over $\mathbb{F}$ (see  \cite[Corollary 2.2.9]{Ox11}).  So, by Remark \ref{rem:matroid h}, we may use J. Huh's result on $h$-vectors of  matroids which are representable over fields of characteristic zero (\cite[Theorem 3]{Hu12}) and conclude  that  $h^A$ is log-concave. 
\end{proof}
The result  which we used to conclude (\cite[Theorem 3]{Hu12})  has been in the meantime generalized by Huh and E. Katz in \cite{HK12}. However,  we find the weaker version  which we cite in the proof better adapted to our setting. \\

The next result shows that, not only is the $h$-vector of a degree matrix with equal rows the Hilbert function of some level algebra, but that the standard determinantal schemes having such a degree matrix are exactly the  level ones.

\begin{proposition}\label{prop:level} Let $X\subseteq\mathbb{P}^{n}$
be a standard determinantal scheme of codimension c, with degree matrix
$A=\left(a_{i,j}\right)\in\mathbb{Z}^{t\times (t+c-1)}$. Then $X$ is level
if and only if $A$ has equal rows.
\end{proposition}

\begin{proof} 
Let $M=\left(f_{i,j}\right)$ be the homogeneous matrix whose maximal minors generate the defining ideal $I_X$ of $X$. Let  $ a_{j}-b_{i} = a_{i,j}=\deg f_{i,j}$, so $M$ defines a graded homomorphism of degree zero
% is standard determinantal, it is generated by the maximal minors of a homogeneous matrix
%, where the $We f_{i,j}$'s  are forms of degree
%$a_{i,j}=a_j-b_i$. 
%%Clearly $M$ defines  a graded homomorphism of degree 0
$$\varphi:F=\bigoplus_{i=1}^tS(b_i)\longrightarrow\bigoplus_{j=1}^{t+c-1}S(a_j)=G.$$
The minimal free resolution of $S/I_X$ is given by the Eagon-Northcott complex with respect to $\varphi$ (see \cite{BV88,MiR08}). Therefore the last free module in it is of the form 
$$F_c= \bigwedge^{t+c-1}G^{*}\otimes S_{c-1}(F)\otimes\bigwedge^{t}F,$$ 
where 
$$
\begin{array}{rclcrcl}
G^{*}&=&{\displaystyle\,\bigoplus_{\phantom{_1\leq k_1\leq}j=1\phantom{\dots k_{c-1}}}^{t+c-1}}S(-a_j)&\quad&
\bigwedge^{t+c-1}G^{*}&=&S\left(-\sum_{j=1}^{t+c-1}a_j\right)\\ 
\rule{0pt}{4ex}S_{c-1}(F)&=&{\displaystyle\bigoplus_{1\leq k_1\leq\cdots\leq k_{c-1}\leq t}}S\left(\sum_{j=1}^{c-1}b_{k_j}\right)&&
\bigwedge^{t}F&=&S\left(\sum_{i=1}^{t}b_i\right)\\
\end{array}
$$

\noindent We can rewrite the shifts in $F_c$ in terms of the entries of $A$ as follows

\begin{align}F_{c}&=\underset{1\leq k_{1}\leq\cdots\leq k_{c-1}\leq
t}{\bigoplus}S(-a_1-\cdots-a_{t+c-1}+b_1+\cdots+b_t+b_{k_1}+\cdots +b_{k_{c-1}}) \notag\\
&=\underset{1\leq k_{1}\leq\cdots\leq k_{c-1}\leq
t}{\bigoplus}S(-a_{k_{1},1}-\cdots-a_{k_{c-1},c-1}-a_{1,c}-\cdots-a_{t,t+c-1})\notag
\end{align}
The scheme $X$ is level if and only if $F_c=S^b(-d)$. In particular the shifts corresponding to the summation indices $(1,t,\ldots,t),\ldots,(t,t,\ldots,t)$  are all equal, that is
$$a_{1,1}+a_{t,2}+\cdots+a_{t,c-1}=\cdots=a_{t,1}+a_{t,2}+\cdots+a_{t,c-1}.$$
This implies that  $a_{1,1}=\cdots=a_{t,1}$, which is equivalent to  the rows of $A$ being  equal. \end{proof}
We believe that, just as in Proposition \ref{prop:level}, an equivalence holds also in Theorem \ref{thm:main}.
\begin{conjecture}\label{conj:main}
If  $A$ is a degree matrix without zeros, then
 $h^A$ is a pure O-sequence if and only if $A$ has equal rows.
\end{conjecture}

The last part of this paper is dedicated to bringing evidence in support of this statement. We first prove that  that Conjecture \ref{conj:main} holds in codimension two.
\begin{proposition}\label{prop:codim2}
If $X\subseteq\mathbb{P}^{n}$ be a codimension 2 standard
determinantal scheme, whose degree matrix
$A\in\mathbb{Z}^{t\times t+1}$
has no zeros, then  $h^{A}$ is a pure O-sequence if and only if $A$ has equal
rows.
\end{proposition}
\begin{proof}%By Theorem \ref{thm:main} we only have to show  the $"\Rightarrow"$ implication. 
Assume that $h^{A}$ is a pure O-sequence and let $B$ be the artinian reduction of $S/I_{X}$.
Then, there exists an artinian monomial level algebra $R/J$, where $R=K[x_1,x_2]$, such that
$h^{A}=HF_{B}=HF_{R/J}$. Since $A$ has no zeros, and we are in
codimension two, by the Hilbert-Burch theorem, (see for instance \cite[Theorem 20.15]{Ei95}) the Hilbert function of $B$ determines uniquely its minimal free resolution, and also the one of $S/I_{X}$. Thus $R/J$ being level implies that also $B$ is level. The claim follows
now from Proposition \ref{prop:level}.  
\end{proof}

\begin{remark}\label{rem:bound} 

If $h=(1,c,h_2,\ldots,h_s)$ is a pure O-sequence, then by counting monomials and divisors of monomials  in each degree, one easily obtains  that 
$$h_{s-i}\leq\min\left\{\binom{c-1+s-i}{c-1}, \,\,h_s\cdot
\binom { c- 1+i}{c-1}\right\}, \quad\forall~i=0,\dots,s.$$
\end{remark}

The next result shows that our conjecture holds when the second-largest entry in the first column of the degree matrix is positive. In particular, it holds for matrices with all entries positive. 

\begin{proposition}\label{prop:ar+1,1>0}
 Let $X\subseteq \mathbb{P}^n$  be codimension $c$ standard determinantal
scheme, whose degree matrix $A=\left(a_{i,j}\right)\in \mathbb{Z}^{t\times (t+c-1)}$ has
$r$ equal maximal rows, with $r<t$,  and no zeros. If  $a_{r+1,1}>0$, then $h^A$ is not a pure O-sequence. 
\end{proposition}
\begin{proof}
 As $a_{r+1,1}>0$, by Remark \ref{rem:compute} and by Lemma \ref{lem:length} we have

 \[h^A_{s-a_{1,1}+a_{r+1,1}}\ge\displaystyle{\binom{r+c-2}{c-1}\cdot\binom{c-1+i}{c-1}+\binom{r+c-3} {c-2 } }. \]
By Lemma \ref{lem:length}  the last entry of $h^A$ is $h_s=\binom{r+c-2}{c-1}$. At the beginning of this section we assumed that $c\ge2$, so Remark \ref{rem:bound} implies that $h^A$ is not a pure O-sequence.
\end{proof}
%
%By formula \alex{name it}, Part $(2)$ of Lemma \ref{lem:length} and  Proposition \ref{prop:lastentries}, we have
%\begin{center}
% $h^A_{s-e_r}=\displaystyle{\binom{r+c-2}{c-1}\cdot\binom{c-1+i}{c-1}+ (t\binom{r+c
%-2} {c-1 } } $
%\end{center}
%Clearly, if $t-r\ge2$,  the entry $h^A_{s-e_r}$ can only increase with respect to $h^{A^{(t,t+c-1)}}_{s'-e_r}$, while the last entry of both $h^A$ and $h^{A^{(t,t+c-1)}}$ is $\binom{r+c-2}{c-1}$.
%
%
%Let now $r<t-1$. Applying  Remark \ref{rem:recursive}\, $(t-r)$-times
%for $a_{t,t+c-1},a_{t-1,t+c-2},\ldots,a_{r+1,r+c}$, we obtain that $h^A$ is
%computed via componentwise addition of vectors $h^1,\ldots,h^{t-(r-1)}$, where
%the last entry $h^{i+1}_{s_{i+1}}$ of $h^{i+1}$ is shifted to the left by
%$e_{r+i}$ places from the last entry of the componentwise sum $h^1+\cdots+h^i$.
%Since $h^1+h^2=h^B$, where $B\in \mathbb{Z}^{(r+1)\times (r+c)}$  is the submatrix
%of $A$ given by $B=[a_{i,j}]_{i=1,\cdots,r+1,j=1,\cdots r+c}$, and as $B$ has
%$r$ maximal rows,the above observation shows that, $h^B$ fails to be a pure
%O-sequence because of the growth at the place $s-e_{r+1}$. Therefore, since
%$e_{r+1}\leq \cdots\leq e_{t}$, $h^A$ can not be a pure O-sequence.
 
Hibi proved in \cite{Hi89} that all pure O-sequences are flawless i.e. $h_i\leq h_{s-i}$ for $i = 0,\dots,\lfloor s/2 \rfloor$.

\begin{proposition}\label{prop:nohibi} Let   $A=(a_{i,j})\in \mathbb{Z}^{t\times (t+c-1)}$ be a degree matrix and  $h^A=(h_0,\ldots,h_s)$ be the corresponding $h$-vector.
If $a_{2,1}<0$,  then there exists an integer $i_0$ such that $h_{i_0}>h_{s-i_0}$. In particular, $h^A$ is not a pure O-sequence.
\end{proposition}

\begin{proof}   
According to Remark \ref{rem:recursive}, 
$h^A_{s-i}=h_{s-i}^{(a_{1,1},\ldots,a_{1,c})},$ for  $i=0,\ldots,(a_{1,1}-a_{2,1}-1).$   
By Proposition \ref{prop:lastentries}, for all $i=0,\ldots,a_{1,2}-1$ we have  
$$h_{s-i}^{(a_{1,1},\ldots,a_{1,c})}=\displaystyle\binom{c-1+i}{c-1}-\binom{c-1+i-a_{1,1}}{c-1}.$$
In particular, as $a_{1,2}-1 = a_{1,1}+a_{2,2}-a_{2,1}-1 > a_{1,1} - a_{2,1} -1\ge a_{1,1}$, we obtain 
$$h_{s-i}^{(a_{1,1},\ldots,a_{1,c})}<\displaystyle\binom{c-1+i}{c-1}, \quad \textup{~for every~}  i = a_{1,1},\ldots,(a_{1,1}-a_{2,1}-1).$$
Thus, as $h^A_{i}=\binom{c-1+i}{c-1}$ for all $i=0,\ldots,\sum_{j=0}^{t}a_{j,j}-1$, every index $ i_0 \in \{a_{1,1},\ldots,a_{1,1}-a_{2,1}-1\}$ satisfies $h_{i_0}>h_{s-i_0}$.
\end{proof}

Propositions \ref{prop:ar+1,1>0} and  \ref{prop:nohibi} have the following direct consequence.
\begin{corollary}
 Conjecture \ref{conj:main} holds for any degree matrix  with only one maximal row.
\end{corollary}
The following examples show that Proposition \ref{prop:nohibi} has no easy generalization to matrices with two or more maximal rows.
\begin{example} 
The matrices $A, B$ and their  upper left $3\times4$ submatrices $A^{(4,5)},
B^{(4,5)}$ show that the conditions $a_{r+1,r}<0$ and $a_{t,t-1}<0$    do not
influence flawlessness. Clearly $h^A$ and
$h^{A^{(4,5)}}$ are flawless, while $h^B$ and $h^{B^{(4,5)}}$ are
not. A quick exhaustive computer search shows that none of the four is a pure
O-sequence. 
\begin{eqnarray*}
A\,\,\,\, &=& \left(\begin{array}{rrrrr} 
2&2&5&5&5\\
2&2&5&5&5\\
-2&-2&\phantom{-}1&\phantom{-}1&\phantom{-}1\\
-2&-2&1&1&1
\end{array}
\right) 
\qquad\qquad 
B \,\,\,\,\,\,\,=\,\,\, \left(\begin{array}{rrrrr} 
1&2&5&5&5\\
1&2&5&5&5\\
-3&-2&\phantom{-}1&\phantom{-}1&\phantom{-}1\\
-3&-2&1&1&1\\
\end{array}
\right)\\
\rule{0pt}{4ex}h^A\,\,&=&(\,1\,,\, 2\,,\, 3\,,\, 4\,,\, 5\,,\, 6\,,\, 4\,,\,4\,,\, 4\,,\, 2\,) \qquad\quad\,\,\,\, h^B \,\,\,\,\,=\,\,\, (\,1\,,\, 2\,,\, 3\,,\, 4\,,\, 5\,,\, 3\,,\, 3\,,\, 3\,,\, 2\,)\\
\rule{0pt}{4ex}h^{A^{(4,5)}}\!\!\!\!\!\!&=&(\,1\,,\, 2\,,\, 3\,,\, 4\,,\, 5\,,\,4\,,\,4\,,\, 4\,,\, 2\,) \qquad\qquad\,\,\,\,\, h^{B^{(4,5)}}\!\!\!=\,\,\, (\,1\,,\, 2\,,\, 3\,,\, 4\,,\, 3\,,\, 3\,,\, 3\,,\, 2\,)
 \end{eqnarray*}
 The matrices $C$ and $D$ below show that  for one maximal  row and all entries positive both situations may appear, namely $h^C$ does not satisfy Hibi's inequalities, while $h^D$ does. By Proposition \ref{prop:ar+1,1>0} non of them is a pure O-sequence.
 \begin{eqnarray*}
C &=& \left(\begin{array}{rrrr} 
3&3&3&3\\
1&1&1&1\\
\end{array}
\right) 
\qquad\qquad\qquad \qquad\quad\,\,\,\,\,\,\,
D \,\,\,=\,\,\, \left(\begin{array}{rrrr} 

2&2&2&2\\
1&1&1&1\\\end{array}
\right)\phantom{23456788762}\\
\rule{0pt}{4ex}h^C\,\,&=&(\,1\,,\, 3\,,\, 6\,,\, 10\,,\, 9\,,\, 7\,,\, 3\,,\,1\,) \qquad\qquad\quad\,\,\,\, h^D \,\,\,\,\,=\,\,\, (\,1\,,\, 3\,,\, 6\,,\, 4\,,\, 1\,)\\
 \end{eqnarray*}
 \end{example}

\bibliographystyle{alpha}
\bibliography{sdv}

\end{document}